\documentclass[reqno]{amsart}
\usepackage{epsfig}
\usepackage{color}
\usepackage{amssymb,amsmath,amsthm,amstext,amsfonts}
\usepackage[utf8]{inputenc}
\usepackage{lineno}
\usepackage{amssymb,amsmath,amsthm,amstext,amsfonts}

\usepackage{color}
\usepackage[mathscr]{eucal}
\usepackage{dsfont}

\usepackage{psfrag}
\usepackage{url}
\usepackage{epstopdf}
\usepackage{epic,eepic}
\usepackage{upgreek}
\usepackage{amssymb}
\usepackage{mathrsfs}
\usepackage{verbatim}
\usepackage{mathrsfs}
\usepackage{amstext}
\usepackage{amsthm}
\usepackage{amssymb}
\usepackage{graphicx}

\usepackage[colorlinks=true, linkcolor=blue, urlcolor=red, citecolor=blue, hyperindex, pagebackref]{hyperref}

\makeatletter \@addtoreset{equation}{section} \makeatother

\renewcommand\thetable{\thesection.\@arabic\c@table}

\theoremstyle{plain}
\newtheorem*{teoA}{Theorem A}
\newtheorem*{teoB}{Theorem B}

\newtheorem{theorem}{Theorem}[section]
\newtheorem{proposition}{Proposition}[section]
\newtheorem{lemma}{Lemma}[section]

\newtheorem{definition}{Definition}[section]
\newtheorem{remark}{Remark}[section]
\newtheorem{example}{Example}[section]

\newcounter{main}

\title
{The Asymptotically Additive Topological Pressure:\\ Variational Principle For Non Compact and Intersection of Irregular Sets}

\author{G. Ferreira}

\address{Giovane Ferreira, Departamento de Matemática, Universidade Federal do Maranhão\\
Av. dos Portugueses, 1966  São Luís - MA, 65065-545}
\email{ferreiras.giovane@gmail.com}

\begin{document}

\begin{abstract}
Let $(X,d,f)$ be a dynamical system, where $(X,d)$ is a compact metric space and $f:X\rightarrow X$ is a continuous map. Using the concepts of \textit{g-almost product property} and \textit{uniform separation property} introduced by Pfister and Sullivan in \cite{Pfister2007}, we give a variational principle for certain non-compact with relation the asymptotically additive topological pressure. We also study the  set of points that are irregular for an collection finite or infinite of asymptotically additive sequences and we show that carried the  full asymptotically additive topological pressure. These results are suitable for systems such as mixing shifts of finite type, $\beta$-shifts, repellers and uniformly hyperbolic diffeomorphisms.
\end{abstract}


\maketitle


\date{\today}



\section{Introduction}

In the present paper we contribute to the theory of multifractal analysis for asymptotically additive sequences of sequentially saturated maps(see Definition \ref{def sequentially saturated}), that include maps that satisfy the specification property(as a large class of uniformly hyperbolic maps). The main result is that the irregular set of maps sequentially saturated for a asymptotically additive potential sequence $\Phi=(\phi_n)_n$ of a topological dynamical system $(X,f)$ carries the full asymptotically additive topological pressure. This generalizes results of \cite{ercai,thompson2012,Thomp2009,Zhao2011} and \cite{Tian2017}. For this purpose, we give a version of the variational principle(in the context ``asymptotically additive") for certain non compacts as in \cite{Pfister2007,Pei10,Thomp2009}. The details are given below.

Let $(X,d,f)$ be a topological dynamical system(TDS), where  $f:X\rightarrow X$ is a continuous map and $X$ is a compact metric space. Let $C^0(f,X)$ denote the space of continuous functions on $X$. A sequence of continuous functions $\Phi:=(\phi_n)_n$ is called \textit{asymptotically additive} for $f$ if for each $\delta>0$, there exists a continuous function $\psi_{\delta}$ such that

\[\limsup_{n\rightarrow \infty}\frac{1}{n}\|\phi_n-S_n\psi_{\delta}\|<\delta,\]
where $\|\cdot\|$ is the supremum norm and $S_n\psi_{\delta}(x)=\sum_{i=0}^{n-1}\psi_{\delta}(f^i(x))$. We denote by $\mathfrak{C}^0(f,X)$   the space of asymptotically additive sequences endowed with the product topology. 


Let $\mathcal{M}(X)$ be the space of Borel probability measures on $X$,  $\mathcal{M}_f(X)\subset \mathcal{M}(X)$ denote the space of $f$-invariant Borel probability measures and $\mathcal{M}_f^e(X)\subseteq \mathcal{M}_f(X)$ the subset of the ergodic measures.

In \cite{Feng2010}, Feng et al defined the Lyapunov exponent of a asymptotically additive sequence $\Phi$ at $x$. For a asymptotically additive sequence $\Phi=(\phi_n)_n$  on $X$ and $x\in X$, the Lyapunov exponent of $\Phi$ at $x$ is the limit(whenever it exists)

\begin{eqnarray}\label{limit lyapunov}
\lambda_{\Phi}(x)=\lim_{n\rightarrow \infty}\frac{\phi_n(x)}{n}.
\end{eqnarray}
By Kingman's sub-additive ergodic theorem, for any $\mu \in \mathcal{M}_f^e(X)$,

\[\lambda_{\Phi}(x)=\Phi^{*}(\mu)\;\;\;\;\;\mbox{for} \;\mu-\mbox{a.e.}\, x\in X,\]
where $\Phi^{*}(\mu):=\lim_{n\rightarrow \infty}\int \frac{\phi_n(x)}{n}d\mu(x)$, that always exists, see Proposition A.1 of \cite{Feng2010}. The authors consider the distribution of the Lyapunov exponents of $\Phi$ on $\alpha$-\textit{level set}  of $\lambda_{\Phi}$. More precisely, for any $\alpha \in \mathbb{R}$ we define the set

\[E({\Phi},f,\alpha)=\{x\in X:\lambda_{\Phi}=\alpha\}.\]
They study the topological entropy $h_{f}(E(\Phi,f,\alpha))$ in the sense of Bowen\cite{Bow73}. If we consider the multifractal decomposition 

\[X=R(\Phi,f)\cup I(\Phi,f), \]
where $R(\Phi,f)=\bigcup_{\alpha \in \mathbb{R}}E(\Phi,f,\alpha)$ is the $\Phi$-regular set($x\in R(\Phi,f)$ is called $\Phi$-regular  point)   and $I(\Phi,f):=\{x \in X: \lim_{n\rightarrow \infty}\frac{\phi_n(x)}{n}\, \mbox{does not exist}\}$ is the $\Phi$-irregular set($x\in I(\Phi,f)$ is called $\Phi$-irregular point), an interest question is about the entropy of the set $I(\Phi,f)$. Here, we address this question. More precisely, we studied the set $I(f):=\bigcup_{\Phi \in \mathfrak{C}^0(X)}I(\Phi,f)$, called of irregular set. As the product $\Pi_{n=1}^{\infty}C^0(f,X)$ is a separable set, then $\mathfrak{C}^0(f,X)\subset \Pi_{n=1}^{\infty}C^0(f,X)$ it is also a separable set, i.e., there exists a enumerable and dense set $\{(\Phi_n)_n;  n \in \mathbb{N}\}$ in $\mathfrak{C}^0(f,X)$ where $ \Phi_n=(\phi^n_k)_{k\in \mathbb{N}}$ such that, by Kingman's Theorem,   $\mu\left(\bigcap_{n=1}^{\infty}R(\Phi_n,f)\right)=1$. Then, for any $\Phi=(\phi_k)_{k\in \mathbb{N}} \in \mathfrak{C}^0(f,X)$ and $\varepsilon>0$, we can choose a $l\in \mathbb{N}$ such that
$$
\|\Phi-\Phi_l\|\leq \varepsilon.
$$
Then, for any $x\in R(f):= \bigcap_{n=1}^{\infty}R(\Phi_n,f)$,
$$
\limsup_{n\rightarrow \infty}\frac{\phi_n(x)}{n}\geq \lim_{n\rightarrow \infty}\frac{\phi_n^l(x)}{n}-\varepsilon
$$
$$
\liminf_{n\rightarrow \infty}\frac{\phi_n(x)}{n}\leq \lim_{n\rightarrow \infty}\frac{\phi_n^l(x)}{n}+\varepsilon
$$
then, $\limsup_{n\rightarrow \infty}\frac{\phi_n(x)}{n}=\liminf_{n\rightarrow \infty}\frac{\phi_n(x)}{n}$. Therefore, $I(f)=X\backslash R(f)$ is a mensurable set.



Various  authors have studied the irregular set and proved, in many cases, this  has full entropy and topological pressure. Pesin and Pitskel \cite{Pesin1984}, showed that it carried the full entropy in the case of the Bernoulli shift on two symbols, Barreira and Schmeling(\cite{Barreira2000}) for case of
generic Hölder continuous function on a conformal
repeller. See more in \cite{li2013,li2014}, \cite{Thomp2010}, \cite{thompson2012}, \cite{BarLiVal} \cite{bomfim2015} and \cite{bomfim2017}, for others authors.

The studies cited above were made for irregular sets under one observable. In \cite{Tian2017}, Tian proposes a study for irregular sets under a collection of observable functions. We make a contribution to the study of irregular sets under a collection finite or infinite of asymptotically additive potentials. More precisely, let $\mathfrak{\hat{C}}^0(f,X):= \{\Phi \in \mathfrak{C}^0(f,X):I(\Phi,f)\neq \emptyset\}$. We prove that the set $I(\Phi,f)$ carries full additive asymptotically topological pressure. In fact, we consider the additive asymptotically topological pressure 
$P_f(\bigcap_{\Psi \in D}I(\Psi,f),\Phi)$, where $D$ is an subset of $\mathfrak{\hat{C}}^0(f,X)$(and even uncountable), see Theorem B below.

Let $Z\subset X$  $f$-invariant Borel set(i.e., $f^{-1}(Z)=Z$), denote by $\mathcal{E}(Z,f)=\{\mu \in \mathcal{M}_f^e(X):\mu(Z)=1\} $ and $\mathcal{M}_xf$ the space of limit measures of the sequence of measures, in weak$^*$ topology,
\[\mathcal{E}_n(x):=\frac{1}{n}\sum_{j=0}^{n-1}\delta_{f^j(x)}.\]

In \cite{Pesin1984}, Theorem A2.1, Pesin and Pitskel proved that for a $f$-invariant Borel set $Z\subset X$, if we consider the Borel $f$-invariant set  $\mathcal{Z}=\{x\in Z:\mathcal{E}(Z,f)\cap \mathcal{M}_xf \neq \emptyset\}$, we have for any continuous function $\varphi \in C(X)$

\[P_f(\mathcal{Z},\varphi)=\sup\{h_{\mu}(f)+\int_{Z}\varphi d\mu: \mu \in \mathcal{E}(Z,f)\}.\]

\begin{definition}\label{def sequentially saturated}
	A TDS $f:X\rightarrow X$ is \textit{sequentially saturated}, if for any $\Phi \in \mathfrak{C}^0(f,X) $ and any compact connected nonempty set $K\subseteq \mathcal{M}_f(X)$ we have
	
	\[P_f(G_K,\Phi)=\inf\{h_{\mu}(f)+\Phi^{*}(\mu): \mu \in K\},\]
	where $G_K=\{x \in X:\mathcal{M}_xf=K\}$.
\end{definition} 
We proved that
\begin{teoB}\label{teo B}
	Let $(X,f)$ be a dynamical system sequentially saturated and assume that $\mathfrak{\hat{C}}^0(f,X)\neq \emptyset$. Then, for any  subset $D\subseteq \mathfrak{\hat{C}}^0(f,X)$ and for each $\Phi \in \mathfrak{C}^0(f,X) $ we have
	
	\[P_f(\bigcap_{\Psi \in D}I(\Psi,f),\Phi)=P_f(\Phi).\]
	
\end{teoB}

A subset $D\subset X$ is saturated if $x\in D$ and the sequences $\mathcal{E}_n(x)$ and $\mathcal{E}_n(y)$ have the same limit-point set, then $y\in D$. Of particular interest are the generic points of $\mu$, i.e. points that satisfy $\lim_{n\rightarrow \infty}\frac{1}{n}\sum_{j=0}^{n-1}\varphi(f^j(x))=\int \varphi d\mu$. We denote the saturated set of generic points of $\mu$ by $G_{\mu}$. In \cite{Bow73} Bowen proved that if $\mu$ is ergodic, then 
\begin{eqnarray}\label{variational bowen}
h_{f}(G_{\mu})=h_{\mu}(f).
\end{eqnarray}
We know that if $\mu\in \mathcal{M}_f(X)$, then $\mu$ is ergodic if and only if $\mu(G_{\mu})=1$. For non-ergodic measures, the Equation (\ref{variational bowen}) cannot hold. It is not difficult give examples such that $G_{\mu}=\emptyset$ and $h_{\mu}(f)>0$.

In \cite{Pfister2007} Pfister and Sullivan treat of the case of non-ergodic measures, introducing two conditions on the dynamics: the 
$g$\textit{-almost product property} and the \textit{uniform separation property}. They proved that

\begin{theorem}\cite{Pfister2007}\label{theo pfister} 
	If the $g$-almost product property and the uniform separation property hold, then for any compact connected non-empty set $K\subset \mathcal{M}_f(X)$
	\[\inf\{h_{\mu}(f):\mu \in K\}=h_{f}(G_K).\]
\end{theorem}
In \cite{Pei10}, Pei and Chen generalized Theorem \ref{theo pfister} for the case of the topological pressure.

\begin{theorem}\label{theo pei}\cite{Pei10}
	If the $g$-almost product property and the uniform separation property hold, then for any compact connected non-empty set $K\subset \mathcal{M}_f(X)$
	\[\inf\{h_{\mu}(f)+\int \varphi d\mu:\mu \in K\}=P_f(\varphi,G_K).\]

\end{theorem}

 We prove a variational principle for the  asymptotically additive topological pressure on certain non-compact sets. We prove that

\begin{teoA}\label{teoA}
	If $f$ satisfies the $g$-almost product property and the uniform separation property, then $f:X\rightarrow X$ is sequentially saturated. 
	
\end{teoA}

\begin{remark}
The proof of Theorem A is given in the subsection  \ref{subsection A} following the lines of \cite{Pfister2007}. We note here that the proof of Theorem A can also be given in a simpler way. We give the details below:

As $\Phi=\{\phi_n\}_n \in \mathfrak{C}^0(f,X)$ is a limit of additive sequences, for every each $k\in \mathbb{N}$, there exist $\psi_k$ such that
$$
\limsup_{n\rightarrow \infty}\frac{1}{n}\|\phi_n-S_n\psi_k\|<\frac{1}{k}.
$$
By Definition \ref{def pressure assinto.} of  asymptotically additive topological pressure  and the definition of topological pressure for continuous maps, it's not hard to see that for $Z\subset X$ we have
$$
\|P(Z,\Phi)-P(Z,\psi_k)\|<\frac{1}{k}\Rightarrow P(Z,\Phi)=\lim_{k\rightarrow \infty}P(Z,\psi_k)
$$
and that for each $\mu \in \mathcal{M}_f(X)$
$$
|\lim_{n\rightarrow \infty}\int \frac{\phi_n}{n}d\mu-\int \psi_k d\mu|<\frac{1}{k} \Rightarrow \Phi^*(\mu)=\lim_{k\rightarrow \infty}\int \psi_k d\mu
$$
By Theorem \ref{theo pei}, for each $k$
$$
P(G_K,\psi_k)=\inf\{h_{\mu}(f)+\int \psi_k d\mu:\mu \in K\}
$$
then
$$
P(G_K,\Phi)=\inf\{h_{\mu}(f)+ \Phi^*(\mu):\mu \in K\}
$$
\end{remark}

In the Section \ref{aplication} we will give some applications and examples of the main results on Cocycles under shift of finite type,  Nonconformal repeller and others.

\section{Preliminaries}
In this section, we remember some concepts and we given some notations for the proof of the Theorem A and Theorem B.




For $a,b \in \mathbb{N}$, $a\leq b$, we denote $[a,b]:=\{c\in \mathbb{N}:a\leq c\leq b\}$ and $\Lambda_n:=[0,n-1]$. The cardinality of a set $\Lambda$ is denoted by $\lvert\Lambda\rvert$ or $\#\Lambda$ .

We set $\langle \phi,\mu \rangle:=\int \phi d\mu$. We define a metric on $\mathcal{M}(X)$ by

$$d(\mu,\nu)=\|\mu-\nu\|:=\sum_{k\geq 1}2^{-k}\lvert\langle \psi_k,\mu-\nu\rangle\rvert.$$
where $\{\psi_k\}_{k\in \mathbb{N}}$ is a countable and dense set of continuous functions on these taking values on $[0,1]$. We use the metric on $X$ given by $d(x,y):=d(\delta_x,\delta_y)$.

We recall the definition of \textit{$g$-almost product property} introduced in \cite{Pfister2007}.
\begin{definition}
	Let $g:\mathbb{N}\rightarrow \mathbb{R}$ be a given nondecreasing unbounded map with the properties 
	\[g(n)<n\;\;\;\; \mbox{and}\;\;\; \lim_{n\rightarrow \infty}\frac{g(n)}{n}=0.\]
	The function $g$ is called blowup function. Let $x\in X$ and $\varepsilon>0$. The $g$-blowup of $B_n(x,\varepsilon)$ is the closed set
	\[B_n(g;x,\varepsilon):=\left\{y\in X: \exists \Lambda\subset \Lambda_n, \lvert\Lambda_n\backslash \Lambda\rvert\leq g(n)\;\mbox{and}\; \max\{d(f^j(x),f^j(y)):j\in \Lambda\leq \varepsilon\right\}.\]
\end{definition}

\begin{definition}
	A TDS $f$ has the $g$-almost product property with blowup function $g$, if there exists a nonincreasing function $m:\mathbb{R}^+\rightarrow \mathbb{N}$, such that for any $k\in \mathbb{N}$, any $x_1\in X,...,x_k\in X$, any positive $\varepsilon_1,...,\varepsilon_k$, and any integers $n_1\geq m(\varepsilon_1),...,n_k\geq m(\varepsilon_k)$, the intersection
	
	\[\bigcap_{j=1}^kf^{-M_{j-1}}B_{n_j}(g;x_j,\varepsilon_j)\neq \emptyset\] where $M_0=0, M_i=n_1+...+n_i$, $i=1,...,k-1$.
\end{definition}
For $\delta>0$ and $\varepsilon>0$, two points $x$ and $y$ are $(\delta,n,\varepsilon)$-separated if

\[\#\{j:d(f^j(x),f^j(y))>\varepsilon,0\leq j\leq n-1\}\geq \delta n.\]
\begin{remark}
	Note that $g$-almost product property is weaker than Bowen's specification property because it requires only partial shadowing of the specified orbit
	segments. All $\beta$-shifts satisfy the $g$-almost product property, see \cite{Pfister2007}.
\end{remark}

A subset $E$ is $(\delta,n,\varepsilon)$-separated if any pair of different points of $E$ are $(\delta,n,\varepsilon)$-separated.
Let $F\subseteq \mathcal{M}(X)$ be a neighborhood of $\nu\in \mathcal{M}_f(X)$, we set $X_{n,F}=\{x\in X:\mathcal{E}_n(x)\in F\}$.

\begin{proposition}\cite{Pfister2005}\label{prop separation}
	Let $\nu \in \mathcal{M}_f(X)$ be ergodic and $h^*<h_{\nu}(f)$. Then there exist $\delta^*>0$ and $\varepsilon^*>0$ so that for each neighborhood $F$ of $\nu$ in $\mathcal{M}(X)$, there exists $n_{F,\nu}^*\in \mathbb{N}$ such that for any $n\geq n_{F,\nu}^*\in \mathbb{N}$, there exists a $(\delta^*,n,\varepsilon^*)$-separated set $\Gamma_n$, such that
	
	\[\Gamma_n\subset X_{n,F} \;\;\mbox{and}\;\; \lvert\Gamma_n\rvert\geq e^{nh^*}.\]
\end{proposition}

Let $N(F,\delta,n,\varepsilon)$ be the maximal cardinality of an $(\delta,n,\varepsilon)$-separated set of $X_{n,F}$ and 
$N(F,n,\varepsilon)$  the maximal cardinality of an $(n,\varepsilon)$-separated set of $X_{n,F}$.

\begin{definition}
	A TDS $f$ has uniform separation property if the following holds. For any $\eta>0$, there exists $\delta^*>0$ and $\varepsilon^*>0$ so that for $\mu$ ergodic and any neighborhood $F\subset \mathcal{M}_f(X)$ of $\mu$, there exists $n_{F,\mu,\eta}^*$ such that, for $n\geq n_{F,\mu,\eta}^*$, we have
	
	\[N(F,\delta^*,n,\varepsilon^*)\geq 2^{n(h_{\mu}(f)-\eta)}.\]
\end{definition}

The previous definition says that the conclusion of Proposition \ref{prop separation} holds uniformly(i.e. $\delta^*, \varepsilon^*$ does not depend of the measure $\nu$) for a TDS with the uniform separation time.

\begin{remark}
	It is easy to see  that uniform separation implies that $h_{top}(f)$ is finite.
	
\end{remark}
\begin{remark}
	In \cite{Pfister2007}, the authors proved that expansive and asymptotically $h$-expansive maps have the uniform separation property.
	
\end{remark}
We define
\[\underline{s}(\nu,\varepsilon)=\inf_{\nu \in F}\liminf_{n\rightarrow \infty}\frac{1}{n}\log N(F,n,\varepsilon)\;\;\;\;\mbox{and}\;\;\;\;\overline{s}(\nu,\varepsilon)=\inf_{\nu \in F}\limsup_{n\rightarrow \infty}\frac{1}{n}\log N(F,n,\varepsilon)\]
where the infimum is take over any base of neighborhood of $\nu$.
Let $\overline{s}(\mu)=\lim_{\varepsilon \rightarrow 0}\overline{s}(\nu,\varepsilon)$ and $\underline{s}(\mu)=\lim_{\varepsilon \rightarrow 0}\underline{s}(\nu,\varepsilon)$.

\begin{proposition}[\cite{Pfister2007}]\label{proposition limit entropy}
	If $\mu \in \mathcal{M}_f(X)$, then $\overline{s}(\mu)\leq h_{\mu}(f)$.
\end{proposition}

\begin{definition}
	The ergodic measures are entropy-dense if for any $\nu \in \mathcal{M}_f(X)$, each neighborhood $F$ of $\nu$, and $h^*<h_{\nu}(f)$, there exists ergodic measure $\rho \in F$ such that $h^*<h_{\rho}(f)$.
\end{definition}

\begin{proposition}[\cite{Pfister2007}]
	Assume that $f$ has the uniform separation property and the ergodic measures are entropy-dense. For any $\eta>0$, there exist $\delta^*>0$ and $\varepsilon^*>0$ so that for $\mu \in \mathcal{M}_f(X)$ and any neighborhood $F\subset \mathcal{M}_f(X)$ of $\mu$, there exists $\eta_{F,\mu,\eta}^*$ such that
	\[N(F;\delta^*,n,\varepsilon^*)\geq e^{n(h_{\mu}(f)-\eta)}\;\;\mbox{if}\;\; n\geq \eta_{F,\mu,\eta}^*. \]
\end{proposition}
\begin{proposition}[\cite{Pfister2007}]
	If the uniform separation property is true and the ergodic measures are entropy-dense, then $s(\mu):=\overline{s}(\mu)=\underline{s}(\mu)$ is well-defined and $s(\mu)=h_{\mu}(f)$ for all $\mu \in \mathcal{M}_f(X)$.
\end{proposition}


In \cite{Zhao2011}, the authors gives a definition of asymptotically additive  topological pressure for an set $Z\subseteq X$ as in \cite{Bar06} and \cite{Mum2006}, motivated by work of \cite{Thomp2010} and \cite{Zhao2008}. Fix $\varepsilon>0$. Let $Z\subset X$, $\Gamma=\{B_{n}(x,\varepsilon)\}_n$ a cover of $Z$ and $\Phi=(\phi_n)_n$  an asymptotically additive sequence. For $\alpha \in \mathbb{R}$  we define the following quantities:

\[Q(Z,\alpha, \Gamma, \Phi):= \sum_{B_{n}(x,\varepsilon)\in \Gamma}e^{-\alpha n+\sup_{x\in B_{n}(x,\varepsilon)}\phi_{n}(x)}\,\;\;\mbox{and}\;\;\]

\[M(Z,\alpha,\varepsilon,N,\Phi)=\inf_{\Gamma}Q(Z,\alpha, \Gamma, \Phi).\]
where the infimum is taken over all covers finite or countable of the form $\Gamma=\{B_{n}(x,\varepsilon)\}_n$ of $Z$ with $n(\Gamma):= \min_n\{n\}\geq N$. We define
\[m(Z,\alpha,\varepsilon,\Phi)=\lim_{N\rightarrow \infty} M(Z,\alpha,\varepsilon,N,\Phi).\]
As the function $M(Z,\alpha,\varepsilon,N,\Phi)$ is non-decreasing on $N$, then the limit always exist. We can show that
\[
P_f(Z,\Phi,\varepsilon):=\inf\{\alpha:m(Z,\alpha,\varepsilon,\Phi)=0\}=\sup\{\alpha:m(Z,\alpha,\varepsilon,\Phi)=\infty\}.\]

\begin{definition}\label{def pressure assinto.}
	The topological pressure of $\Phi$ on $Z$ is given by	
$$
P_f(Z,\Phi)=\lim_{\varepsilon \rightarrow 0}P_f(Z,\Phi, \varepsilon).$$
\end{definition}
A sequence $\Phi$ is an almost-additive sequence if 
\[\phi_n(x)+\phi_m(f^n(x))-C_{\Phi}\leq \phi_{n+m}(x)\leq \phi_n(x)+\phi_m(f^n(x))+C_{\Phi},\]
for all $x\in X$, $n,m \in \mathbb{N}$ and some constant $C_{\Phi}$.
\begin{remark}
	
	Feng and Huang proved, in \cite{Feng2010}, that an almost-additive sequence is indeed asymptotically additive. Then, by Proposition 4.7 in \cite{Cao08} we have  variation principle for topological pressure of asymptotically additive sequence:
\begin{eqnarray}\label{eq variational principle}
P_f(\Phi):=P_f(X,\Phi)=\sup\{h_{\mu}(f)+\Phi^{*}(\mu):\mu \in \mathcal{M}_f(X)\}.\end{eqnarray}
	
	\begin{remark}
		A asymptotically additive sequence may not be  almost-additive sequence. The next examples illustrate this.
	\end{remark}
	
	\begin{example}\label{Ban-Cao-Hu2010}\cite{Ban-Cao-Hu2010}
		Let a $C^1$ map $f:M\rightarrow M$ defined on $C^{\infty}$ $m$-dimensional Riemannian manifold. Let $\Lambda\subset M$ is a compact $f$-invariant subset. The set $\Lambda$ is called an average conformal repeller if for any $f$-invariant ergodic measure $\mu$ the Lyapunov exponents of $\mu$ $\lambda_i(\mu)$, $i=1,...,m$ are equal and positive. The authors show that the limit
		
		\[\lim_{n\rightarrow \infty}\log\frac{\|Df^n(x)\|}{\|Df^n(x)^{-1}\|^{-1}}=0\]
		converges uniformly on the average conformal repeller $\Lambda$. It is not hard to check hat the sequences  $\Phi_1=(\log \|Df^n(x)^{-1}\|^{-1})_n$ or $\Phi_2=(\log \|Df^n(x)\|)_n$ are asymptotically additive but, these may not be almost additive.
	\end{example}
	
	\begin{example}\label{Barreira1996}\cite{Bar06}
		Let $(X,f,d)$ be a TDS. Assume that $f$ is expanding on $X$, in the sense that there exists constants $a\geq b\geq 1$ and $\varepsilon_0>0$ such that
		
		\[B(f(x),b\varepsilon)\subset f(B(x,\varepsilon))\subset B(f(x),a\varepsilon)\;\;\;\; \mbox{for all}\; x\in X \,\, \mbox{and}\,\, 0<\varepsilon<\varepsilon_0.\]
		Then, $X$ has a Markov partition $R_1,..,R_m$ of arbitrarily small diameters, see \cite{Bar06}. We define a $m \times m$ matrix $A=(a_{ij})$ with $a_{ij}=1$ if $R_i\cap f^{-1}(R_j)\neq \emptyset$ and $a_{ij}=0$ otherwise. For each $\omega=(i_1,i_2,...)\in \Sigma_A^+$, $n\geq 1$ and $k\geq 0$, put
		
		\[\underline{\lambda}_k(\omega,n)=\min \inf\left\{\frac{d(x,y)}{d(f^n(x),f^n(y)}:x,y \in R_{j_1...j_{n+k}}\,\,\mbox{and}\,\, x\neq y\right\}\,\,\mbox{and}\]
		\[\overline{\lambda}_k(\omega,n)=\max \sup\left\{\frac{d(x,y)}{d(f^n(x),f^n(y)}:x,y \in R_{j_1...j_{n+k}}\,\,\mbox{and}\,\, x\neq y\right\},\]
		where $R_{j_1...j_n}=\{x\in X: f^{i-1}(x)\in R_{j_i}\,\,\mbox{for}\,\,i=1,...,n\}$ is the cylinder of length $n$, and the maximum and minimum are taken over the $j_1...j_{n+k}\in \Sigma_A^+$ such that $j_1...j_{n}=(i_1...i_n)$. We say that $f$ is asymptotically conformal if there exists $k\geq 0$ such that the limit
		\[\lim_{n\rightarrow\infty} \frac{1}{n}\log\frac{\overline{\lambda}_k(\omega,n)}{\underline{\lambda}_k(\omega,n)}=0,\]
		converges uniformly on $\Sigma_A^+$. As $\Phi_1=(\log\overline{\lambda}_k(\omega,n))_n$ is sub-additive and $\Phi_2=(\log\underline{\lambda}_k(\omega,n))_n$ is a sup-additive its easy to see that $\Phi_1=(\log\overline{\lambda}_k(\omega,n))_n$ or $\Phi_2=(\log\underline{\lambda}_k(\omega,n))_n$ are asymptotically additive, but may not are almost additive.
	\end{example}

\end{remark}




\section{Proofs}

\subsection{Proof of Theorem A}\label{subsection A}
In this subsection we will prove Theorem A. We will give some results that help us in the proof.
\begin{proposition}[Corollary 3.1 of \cite{Pfister2007}]\label{prop. separation uniform}
	Assume that $(X,d,T)$ has the uniform separation property, and that the ergodic measures are entropy dense. For any $\eta>0$, there exist $\delta^*>0$ and $\varepsilon^*>0$ so that for $\mu \in \mathcal{M}_f(X)$ and any neighborhood $F\subset \mathcal{M}(X)$ of $\mu$, there exists $n_{F,\mu,\eta}$ such that 
	
	\[N(F;\delta^*,n,\varepsilon^*)\geq e^{n(h_{\mu}(f)-\eta)}\;\;\;\mbox{if}\;\;\; n\geq n_{F,\mu,\eta}. \]
	For any $\mu \in \mathcal{M}_f(X)$,
	\[h_{\mu}(f)\leq \lim_{\varepsilon \rightarrow 0}\lim_{\delta \rightarrow 0}\inf_{\mu \in F}\liminf_{n\rightarrow \infty}\frac{1}{n}\log N(F;\delta,n,\varepsilon).\]
\end{proposition}
We define $^K G:=\{x \in X; \{\mathcal{E}_n(x)\}_n \;\mbox{has a limit-point on $K$}\}$. 

\begin{theorem}\label{thm specification/saturated}
	Let $(X,d,f)$ be  a TDS and $K\subset \mathcal{M}_f(X)$ be compact subset. Then,
	
	\[P_f(^K G,\Phi)\leq \sup\{h_{\mu}(f)+\Phi^{*}(\mu):\mu \in K\}.\]
	Consequently, $P_f(G_K,\Phi)\leq \inf\{h_{\mu}(f)+\Phi^{*}(\mu):\mu \in K\}$.
	
\end{theorem}

\begin{proof}
	
	Let $K\subset \mathcal{M}_f(X)$ be compact subset and $\mu \in \mathcal{M}_f(X)$. Put $s:=\sup\{h_{\theta}(f)+\Phi^{*}(\theta),\,\theta \in K\}$. Assume, without loss of generality, that $s<\infty$. Let $s'=s+2\delta$, with $\delta>0$.

For a neighborhood $F$ of $\mu$, denoted by $F_{\mu}$, the function $N(F_{\mu},n,\cdot)$ be non increasing, by Proposition \ref{proposition limit entropy} we have
	\[\inf_{F_{\mu}}\limsup\frac{1}{n}\log N(F_{\mu},n,\varepsilon)\leq h_{\mu}(f)\,\,\; \mbox{for all}\;\; \varepsilon>0.\]
	Then, for $\varepsilon>0$, there exist a neighborhood $B(\mu,\zeta_{\varepsilon})\subseteq F_{\mu}$ of $\mu$ and a number $M(B(\mu,\zeta_{\varepsilon}),\varepsilon)$ such that
	
	\[\frac{1}{n}\log N(B(\mu,\zeta_{\varepsilon}),n,\varepsilon)\leq h_{\mu}(f)+\delta,\]
	for each $n\geq M(B(\mu,\zeta_{\varepsilon}),\varepsilon)$.
	
	Let $E$ be a maximal $(n,\zeta_{\varepsilon})$-separated of $X_{n,B(\mu,\zeta_{\varepsilon})}$(which also is a $(n,\zeta_{\varepsilon})$-spanning of $X_{n,B(\mu,\zeta_{\varepsilon})}$) with cardinality $N(B(\mu,\zeta_{\varepsilon}),n,\varepsilon)$. Then,
	
	\begin{eqnarray*}
		M(X_{n,B(\mu,\zeta_{\varepsilon})},s',\Phi,n,\varepsilon)&\leq& \sum_{B_n(x,\varepsilon)\in \Gamma}e^{-s'n+\sup_{y\in B_n(x,\varepsilon)}\phi_n(y)},
	\end{eqnarray*}
	where $\Gamma:=\cup_{x\in E}B_n(x,\varepsilon/3)\supseteq X_{n,B(\mu,\zeta_{\varepsilon})}$.
	
	For $x\in E\subseteq X_{n,B(\mu,\zeta_{\varepsilon})}$ and since $\Phi$ is asymptotically additive, for $\zeta_{\varepsilon}>0$ there exists $\varphi_{\zeta_{\varepsilon}}:=\varphi_{\zeta}\in C^0(X)$ and $n_0$ such that for all $n>n_0$, we have
	\begin{eqnarray*}
		\lvert\frac{1}{n}\phi_n(x)-\Phi^*(\mu)\rvert&\leq& \lvert\frac{1}{n}\phi_n(x)-\frac{1}{n}S_n \varphi_{\zeta}(x)\rvert+\lvert\frac{1}{n}S_n \varphi_{\zeta}(x)-\int \varphi_{\zeta}d\mu\rvert\\
		&+&\lvert\int \varphi_{\zeta}d\mu-\Phi^*(\mu)\rvert
		=\lvert\frac{1}{n}\phi_n(x)-\frac{1}{n}S_n \varphi_{\zeta}(x)\rvert\\
		&+&\lvert\int \varphi_{\zeta}(x)-\int \varphi_{\zeta}d\mu\rvert
		+\lvert\int \varphi_{\zeta}d\mu-\Phi^*(\mu)\rvert
		\leq 3\zeta_{\varepsilon}
	\end{eqnarray*}
	and $\sup_{y\in B_n(x,\varepsilon)}\phi_n(y)\leq n(2\zeta_{\varepsilon}+\lvert\varphi_{\zeta}\rvert_{\varepsilon})+\phi_n(x)$, where $\lvert\varphi_{\zeta}\rvert_{\varepsilon}:=\sup\{\lvert\varphi_{\zeta}(x)-\varphi_{\zeta}(y)\rvert: d(x,y)<\varepsilon\}$.
	Then,
	
	\begin{eqnarray*}
		M(X_{n,B(\mu,\zeta_{\varepsilon})},s',\Phi,n,\varepsilon)&\leq& \sum_{B_n(x,\varepsilon)\in \Gamma}e^{n(-s'+\Phi^*(\mu)+\lvert\varphi_{\zeta}\rvert_{\varepsilon}+5\zeta_{\varepsilon})}\\
		&\leq& N(B(\mu,\zeta_{\varepsilon}),n,\varepsilon)e^{n(-s'+\Phi^*(\mu)+\lvert\varphi_{\zeta}\rvert_{\varepsilon}+5\zeta_{\varepsilon})}\\
		&\leq& e^{n(-\delta+\lvert\varphi_{\zeta}\rvert_{\varepsilon}+5\zeta_{\varepsilon})}.
	\end{eqnarray*}
	For a fixed $\delta$, there exists $\varepsilon_0$ such that $-\delta+\lvert\varphi_{\zeta}\rvert_{\varepsilon}+5\zeta_{\varepsilon}<0$, when $\varepsilon<\varepsilon_0$.
	
	As $K$ is compact set, given a fixed $\varepsilon>0$, we can choose a finite open cover $\{B(\mu_j,\zeta_{\varepsilon})\}_j$, with $j=1,...,m_{\varepsilon}$ of $K$. For sufficiently large $M>0$, $\bigcup_{n\geq M}\bigcup_{j=1}^{m_{\varepsilon}}X_{n,B(\mu_j,\zeta_{\varepsilon})}$ is a cover of $^KG$. Then, for $M\geq \max_{j}M(B(\mu_j),\varepsilon)$, we have
	
	\[M(^KG,s',\Phi,n,\varepsilon)\leq \sum_{n\geq M}\sum_{j=1}^{m_{\varepsilon}}e^{n(-\delta+\lvert\varphi_{\zeta}\rvert_{\varepsilon}+5\zeta_{\varepsilon})}\leq m_{\varepsilon}\sum_{n\geq M}e^{n(-\delta+\lvert\varphi_{\zeta}\rvert_{\varepsilon}+5\zeta_{\varepsilon})}.\]
	When $M\rightarrow \infty$, we have $m_{\varepsilon}\bigcup_{n\geq M}e^{n(-\delta+\lvert\varphi_{\zeta}\rvert_{\varepsilon}+5\zeta_{\varepsilon}))}\rightarrow 0$. Thus, $P(^KG, \Phi, \varepsilon)\leq s$ for all $\varepsilon<\varepsilon_0$. Therefore,
	
	\[P_f(^KG,\Phi)=\lim_{\varepsilon\rightarrow 0}P_f(^KG, \Phi, \varepsilon)\leq \sup_{\mu \in K}\{h_f(\mu)+\Phi^*(\mu)\}.\]
	
	For other part, note that $G_K\subset ^{\{\mu\}}G$ for all $\mu \in K$. Then,
	
	\[P_f(G_K,\Phi)\leq P_f(^{\{\mu\}}G,\Phi)\leq h_{\mu}(f)+\Phi^{*}(\mu)\;\;\mbox{for all}\;\; \mu \in K \]
	and $P_f(G_K,\Phi)\leq \inf\{h_{\mu}(f)+\Phi^{*}(\mu):\mu \in K\}$.
\end{proof}

	
	

A ball in $\mathcal{M}(X)$ is denoted by
$$
B(\nu, \xi);=\{\alpha \in \mathcal{M}(X): d(\alpha,\nu)<\xi\}
$$

\begin{proposition}[Lemma 2.1, \cite{Pfister2007}]\label{aprox. measures}
	Assume that $(X,d,f)$ has a $g$-almost product property. Let $x_1,...,x_k \in X$, $\varepsilon_1>0,...,\varepsilon_k>0$ and $n_1\geq m(\varepsilon_1)$,...,$n_k\geq m(\varepsilon_k)$ be given. Assume that 
	\[\mathcal{E}_{n_j}(x_j)\in B(\nu_j,\xi_j),\;\;\mbox{for}\;\; j=1,...,k.\]
	Then for any $y\in \cap_{j=1}^k f^{-M_{j-1}}B_{n_j}(g;x_j,\varepsilon_j)$ and any Borel probability measure $\alpha$ on $X$ we have that
	\[d(\mathcal{E}_{M_k}(y),\alpha)\leq \sum_{j=1}^{k}\frac{n_j}{M_k}(\xi_j'+d(\nu_j,\alpha)),\]
	where $M_i=n_1+...+n_i$ and $\xi_i'=\xi_i+\varepsilon_i+\frac{g(n_i)}{n_i}$, $i=1,...,k$.
\end{proposition}

\begin{lemma}[\cite{Pfister2003}]
	Let $K\subset \mathcal{M}_f(X)$ be a compact connected non-empty set. Then there exists a sequence $\{\alpha_1,...,\alpha_k,...\}\subset K$ such that for each $n\in \mathbb{N}$
	
	\[\overline{\{\alpha_j:j\in \mathbb{N}, j>n\}}=K \;\;\mbox{and}\;\; \lim_{j\rightarrow \infty}d(\alpha_j,\alpha_{j+1})=0.\]
\end{lemma}
\begin{theorem}\label{theorem B2}
	Let $(X,d,f)$ be a TDS with the g-almost product property and uniform separation. Let $K$ be a connected non-empty compact subset of $\mathcal{M}_f(X)$. Then
	
	\[P_f(G_K,\Phi)\geq \inf\{h_{\mu}(f)+\Phi^*(\mu):\mu \in K\}.\]
\end{theorem}

The proof follow  the line in \cite{Pfister2007} and \cite{Pei10} with some changes. We shall repeat here part of the arguments.

Let $\eta>0$, and $h^*:= \inf\{h_{\mu}(f)+\Phi^*(\mu):\mu \in K\}-\eta$. For $s<h^*$, we put $h^*-s=2\delta$. By Proposition \ref{prop. separation uniform}, we can find $\delta^*>0$ and $\varepsilon^*>0$ such that for each neighborhood $F$ of $\mu$  there exists $n_{F,\mu,\eta}^*$ with

\begin{eqnarray}\label{equation cardinality}
N(F;\delta^*,n,\varepsilon^*)\geq e^{n(h_{\mu}(f)-\eta)}\;\; \mbox{for all}\;\; n\geq n_{F,\mu,\eta}^*.
\end{eqnarray}	
Take $(\varepsilon_k)_k$ and $(\xi_k)_k$ two decreasing sequences converging for zero such that $\varepsilon<\varepsilon^*$ and  $\lvert\int \varphi_{\delta}d_{\mu}-\int\varphi_{\delta}d\alpha_k\rvert<\frac{\delta}{18}$, for all $\mu \in B(\alpha_k,\xi_k+2\varepsilon_k)$ where $\varphi_{\delta}$ is a continuous functions such that $\limsup_{n\rightarrow \infty}\frac{1}{n}\|\phi_n-S_n\varphi_{\delta}\|<\frac{\delta}{18}$. By Equation (\ref{equation cardinality}), there exists $n_k$ and a $(\delta^*,n_k,\varepsilon^*)$-separated sets $\Gamma_k\subset X_{n_k,B(\alpha_k,\xi_k)}$ with cardinality $\lvert\Gamma_k\rvert\geq e^{n_k (n_k(h_{\alpha}(f)-\eta))}$. We can assume that the sequence $n_k$ satisfies 
\begin{equation}\label{eq.error function}
\delta^*n_k>2 g(n_k)+1 \;\; \mbox{and} \;\; \frac{g(n_k)}{n_k}\leq \varepsilon_k.
\end{equation}
By Proposition \ref{aprox. measures} and Equation (\ref{eq.error function}) we have that if $x\in \Gamma_k$, $y\in B_{n_k}(g;x,\varepsilon_k)$, then

\begin{equation}\label{eq. inclusion ball}
\mathcal{E}_{n_k}(y)\in B(\alpha_k,\xi_k+2\varepsilon_k).
\end{equation}
We choose a strictly increasing sequence of positive integers numbers $(N_k)_k$ satisfying 

\[n_{k+1}\leq \xi_k\sum_{j=1}^{k}n_jN_j\;\;\mbox{and}\;\; \sum_{j=1}^{k-1}n_jN_j\leq \xi_k\sum_{j=1}^{k}n_jN_j,\]	
and we define the sequences $(n_j'), (\alpha_j'), (\varepsilon_j')$ and $(\Gamma_j')$ by setting for  $j=N_1+...+N_{k-1}+q$, with $1\leq q\leq N_k$, $n_j':=n_k$, $\varepsilon_j':=\varepsilon_k$, $\xi_{j'}:=\xi_{k}$ and $\Gamma_j':=\Gamma_k$.

Let
\begin{equation*}
G_k:=\bigcap_{j=1}^k\left(\bigcup_{x_j.\in \Gamma_j'
}f^{-M_{j-1}}B_{n_j'}(g;x_j,\varepsilon_j')\right)\;\;\;\mbox{with}\;\;\; M_j:=\sum_{l=1}^{j}n_l'.
\end{equation*} 
Note that $G_k$ is a non-empty closed set. Each element of $G$ can be indicated by $(x_1,...,x_k)$, where $x_j\in \Gamma_j'$. The proof of Theorem \ref{theorem B2} follow the Lemmas \ref{lemma pfister} and \ref{lemma key} below.

\begin{lemma}[\cite{Pfister2007}]\label{lemma pfister}
	Let $\varepsilon>0$ be such that $4\varepsilon=\varepsilon^*$ and $G:=\bigcap_{k\geq 1}G_k$.
	
	\begin{enumerate}
		\item Let $x_j,y_j \in \Gamma_j'$ with $x_j\neq y_j$. If $x\in B_{n_j'}(g;x_j,\varepsilon_j')$ and $y\in B_{n_j'}(g;y_j,\varepsilon_j')$, then 
		\[\max\{d(f^m(x),f^m(y):0\leq m\leq n_j-1\}>2\varepsilon.\]
		
		\item $G$ is a closed set, which is the union of non-empty sets $G(x_1,x_2,...)$ where $x_j\in \Gamma_j'$.
		
		\item $G\subset G_K$.
	\end{enumerate}
\end{lemma}

\begin{lemma}\label{lemma key}
	$P_f(G,\Phi)\geq h^*$.
\end{lemma}
\begin{proof}
	We have $\frac{M_n}{M_{n+1}}\rightarrow 1$. Let $\Gamma_k$ be a $(\delta^*,n_k,\varepsilon^*)$-separated set of $X_{n_k, B(\alpha_k,\xi_k)}$ such that $\lvert\Gamma_k\rvert\geq e^{n_k(h_{\alpha_k}(f)-\eta)}$. We prove that $M(G,s,\Phi,n,\varepsilon)\geq 1$. Note that for each $^k x \in \Gamma_k$ we have $\mathcal{E}_{n_k}(^k x)\in B(\alpha_k,\xi_k)$. Then, for $n_k$ sufficiently large
	\begin{eqnarray*}
		\left\lvert\frac{\phi_{n_k}(^k x)}{n_k}-\Phi^*(\mu)\right\rvert&\leq& \left\lvert\frac{\phi_{n_k}(^k x)}{n_k}-\frac{S_{n_k}\varphi_{\delta}(^k x)}{n_k} \right\rvert+\left\lvert\frac{S_{n_k}\varphi_{\delta}(^k x)}{n_k}-\int \varphi_{\delta}d\alpha_k \right\rvert\\
		&+&\left\lvert\int \varphi_{\delta}d\alpha_k-\Phi^*(\mu)\right\rvert \leq \frac{\delta}{18}+\frac{\delta}{18}+\frac{\delta}{18}=\frac{\delta}{6}.
	\end{eqnarray*}
	Then,
	\begin{eqnarray*}
		\lvert\Gamma_k\rvert\geq e^{n_k\left(h_{\alpha_k}(f)+\Phi^*(\mu)-\eta\right)-\phi_{n_k}(^k x)-n_k\delta/6}\geq e^{n_kh^*-\phi_{n_k}(^k x)-n_k\delta/6}.
	\end{eqnarray*}
	As $G$ is compact set, we can choose finite covers $\mathcal{C}=\{B_m(x,\varepsilon)\}_{m}$ of $G$ such that $B_m(x,\varepsilon)\cap G\neq \emptyset$ for all $B_m(x,\varepsilon)\in \mathcal{C}$. For each $\mathcal{C}\in \mathcal{G}_n(G,\varepsilon)$, where $\mathcal{G}_n(G,\varepsilon)$ is the collection of all finite or countable covers of $G$ by
	sets of the form $B_m(x,\varepsilon)$ with $m \geq n$(for $n$ sufficiently large), we define the cover $\mathcal{C}'$ where we replace each ball
	$B_m(x,\varepsilon)$ by $B_{M_p}(x,\varepsilon)$ when $M_p\leq m < M_{p+1}$. Then
	
	\begin{eqnarray*}
		M(G,s, \Phi, n,\varepsilon)&=& \inf_{\mathcal{C}\in \mathcal{G}_n(G,\varepsilon)}\sum_{B_{m}(x,\varepsilon)\in \mathcal{C}}e^{-s m+\sup_{z \in B_{m}(x,\varepsilon)}\phi_{m}(z)} \\
		&\geq& \inf_{\mathcal{C}\in \mathcal{G}_n(G,\varepsilon)}\sum_{z\in B_{M_p}(x,\varepsilon)\cap G}e^{-s m+\phi_{m}(z)}.
	\end{eqnarray*}
	Consider $\mathcal{C}'$ and let $m$ be the largest such that there exist $B_{M_p}(x,\varepsilon)\in \mathcal{C}'$. We put	\[\mathcal{W}_k=\prod_{i=1}^{k}\Gamma'_i\;\;\mbox{and}\;\; \mathcal{\overline{W}}_m:=\bigcup_{k=1}^m\mathcal{W}_k.\]
	Each $z \in B_{M_p}(x,\varepsilon)\cap G$ corresponds to a point in $\mathcal{W}_p$ that is uniquely defined(Lemma \ref{lemma pfister}(i)). The word $v\in \mathcal{W}_j $ is a prefix of $w\in \mathcal{W}_k$ if the first $j$ entries of $w$ coincide with $v$. If $\mathcal{W}\subset \mathcal{\overline{W}}_m$ contains a prefix of each word of $\mathcal{W}_m$, then
	\[\sum_{k=1}^{m}\lvert\mathcal{W}\cap \mathcal{W}_k\rvert\lvert\mathcal{W}_k\rvert/\lvert\mathcal{W}_k\rvert\geq \lvert\mathcal{W}_k\rvert.\]
	Thus if $\mathcal{W}$ contains a prefix of each word of $\mathcal{W}_m$, 
	\[\sum_{k=1}^{m}\lvert\mathcal{W}\cap \mathcal{W}_k\rvert\lvert\mathcal{W}_k\rvert\geq 1.\]
	It is easy to verify that
	
	\[\lvert\mathcal{W}_p\rvert\geq e^{M_p h^*-\sum_{i=1}^{p}\left(\phi_{n_i}(^ix)+n_i'\delta\right)}.\]
	Therefore, $\sum_{B_{M_p}(x,\varepsilon)\in \mathcal{C}'}e^{-M_p h^*+\sum_{i=1}^{p}\left(\phi_{n_i}(^ix)+n_i'\delta\right)}\geq 1$
	where $^i x \in \Gamma_i'$. We prove that 
	\[M_ph^* -\sum_{i=1}^{p}(\phi_{n_i}(^{i'}x)+n_i'\delta/6)-sm+\phi_m(z)>0,\;\; \mbox{for}\;\; z\in G.\]
	By definition of asymptotically  additive we have, for $\delta>0$, that there exist $\varphi_{\delta}$ and $n_0$ such that for all $m,n_i'>n_0$ with $i=1,...,p$, we have
	
	\[-\phi_{n_i'}(^ix)+S_{n_i'}\varphi_{\delta}(^ix)>-n_i'\frac{\delta}{18} >-n_i'\frac{\delta}{6} \;\; \mbox{and}\;\; \phi_{m}(z)-S_{m}\varphi_{\delta}(z)> -m\frac{\delta}{6}. \]Then
	
	\begin{eqnarray*}
		&&\!\!\!\!\!\!\!\!\!\!\!m(h^*-s)-\sum_{i=1}^{p}(\phi_{n_i}(^{i}x)+n_i'\delta/6)+\phi_m(z)-(m-M_p)h^*\\
		&>&
		m(h^*-s-\frac{\delta}{6})-\sum_{i=1}^{p}\left(S_{n_i'}\varphi_{\delta}(^ix)+n'_i\frac{\delta}{3}\right)
		+ S_m\varphi_{\delta}(^ix)-(m-M_p)h^*\\
		&=&m(2\delta-\frac{\delta}{6})-M_{p}\frac{\delta}{3}+\sum_{i=1}^{p}\left(S_{n_i'}\varphi_{\delta}(f^{M_{i-1}}(z))-S_{n_{i}'}\varphi_{\delta}(^i x)\right)
		-(m-M_p)h^*\\
		&+& S_{m-M_p}\varphi_{\delta}(f^{M_p}(z)).
	\end{eqnarray*}	
	And as 
	$$G=\bigcap_{k\geq 1}\left(\bigcap_{j=1}^k\left(\bigcup_{x_j\in \Gamma_j'}f^{-M_{j-1}}B_{n_j'}(g;x_j,\varepsilon_j')\right)\right),$$ 
	we have that there exists $j$ such that $f^{M_{j-1}}(z) \in B_{n_j'}(g;x_j,\varepsilon_j')$. By Inclusion (\ref{eq. inclusion ball}) and $^ix \in \Gamma_i'$ we have
	
	\[\mathcal{E}_{n_i'}(f^{M_{i-1}}(z)) \in \mathcal{B}(\alpha,\xi_i'+2\varepsilon_i')\;\;\;\;\mbox{and}\;\;\; \mathcal{E}_{n_i'}(^ix)\in \mathcal{B}(\alpha,\xi_i').\] then
	
	\[\left\lvert\int \varphi_{\delta}d\mathcal{E}_{n_i'}(f^{M_{i-1}}(z))-\int \varphi_{\delta}d\mathcal{E}_{n_i'}(^ix)\right \rvert n_i'=\left\lvert S_{n_i'}\,\varphi_{\delta}(f^{M_{i-1}}(z))-S_{n_i'}\,\varphi_{\delta}(^ix)\right\rvert \leq n_i'\frac{\delta}{2}.\]
	Thus,
	
	\begin{eqnarray*}
		M_ph^*-\sum_{i=1}^{p}(\phi_{n_i}(^{i}x)+n_i'\delta/6)-sm+\phi_m(z)&\geq& m(2\delta-\frac{\delta}{6})-M_p(\frac{\delta}{2}+\frac{\delta}{3})\\
		&-&n_{p+1}'(\|g_k\|+h^*)\\
		&\geq&(2\delta-\frac{\delta}{6}-\frac{5\delta}{6})M_p -n_{p+1}'(\|g_k\|+h^*)\\
		&=&M_p\delta-n_{p+1}'(\|g_k\|+h^*).
	\end{eqnarray*}
	we can choose $p$ such that $M_p\delta-n_{p+1}'(\|g_k\|+h^*)>0$, because $\lim_{p\rightarrow \infty}\frac{n_{p+1}'}{M_p}=0$. Then,
	
	\begin{eqnarray*}
		M(G;s,\Phi,n,\varepsilon)\geq \inf_{\mathcal{C}\in \mathcal{G}_n(G,\varepsilon)}\sum_{B_{M_p}(x,\varepsilon)\in \mathcal{C}'}e^{M_ph^*-\sum_{i=1}^{p}(\phi_{n_i}(^{i}x)+n_i'\delta/6)}\geq 1
	\end{eqnarray*}
	this implies that $P(G,\Phi,\varepsilon)\geq s$. By fact that $s<h^*$ and the arbitrary choice of $\eta$ the Lemma \ref{lemma key} it's proved. 
	
	As $G\subset G_K$ then $P_f(G,\Phi)\leq P_f(G_K,\Phi)$. Therefore, we have proved Theorem \ref{theorem B2}.

\end{proof}

\subsection{Proof of Theorem B}

In this subsection, we proved Theorem B. We will use the following Lemma:

\begin{lemma}\label{lemma.inf.sup}

	Let $(X,f)$ be a dynamical systems,  $\Phi:=(\phi_n)_n \in \mathfrak{C}^0(f,X)$ and $x \in X$. Then
	\[\Phi \in \hat{\mathfrak{C}}^0(f,X),\;\; x \in I(\Phi,f)\;\;\mbox{iff} \inf_{\mu \in \mathcal{M}_xf}\Phi^{*}(\mu)< \sup_{\mu \in \mathcal{M}_xf}\Phi^{*}(\mu).\]
\end{lemma}
\begin{proof}
	Assume that $\Phi \in \hat{\mathfrak{C}}^0(f,X) $ and $x \in I(\Phi,f)$. Then, $\frac{\phi_n(x)}{n}$ not converge pointwise for a constant, in particular, there exist a sequence $n_i$ converging for $+\infty$ and $\varepsilon_0>0$ such that
	\[\left\lvert\frac{1}{n_i}\phi_{n_i}(x)-\Phi^*(\mu_1)\right\rvert> \varepsilon_0\] where $\mu_1 \in \mathcal{M}_x(f)$;
	
	Take a subsequence(if necessary) of the sequence 
	\[\nu_{i}:=\left(\frac{1}{n_i}\sum_{i=0}^{n_i-1}\delta_{f^j(x)}\right)_i\]
	such that $\nu_i \rightarrow \mu_2 $ in weak$^*$ topology. Obviously $\mu_2 \in\mathcal{M}_xf$. As $\Phi$ is asymptotically additive, for each $k$, there exists $\psi_{1/k}:=\psi_k \in C^{0}(X)$ such that 
	$\limsup_{n\rightarrow \infty}\frac{1}{n}\|\phi_n - S_n\psi_k\|<\frac{1}{k} $.
	
	Therefore, for $i$ sufficiently large
	$$\frac{S_{n_i}\psi_k(x)}{n_i}-\frac{1}{k}\leq\frac{\phi_{n_i}(x)}{n_i} \leq \frac{S_{n_i}\psi_k(x)}{n_i}+\frac{1}{k}.$$
	
	Then, we have
	\[\lim_{i\rightarrow \infty}\int \psi_kd\nu_{i}-\frac{1}{k}\leq \liminf_{i\rightarrow \infty}\frac{\phi_{n_i}(x)}{n_i} \leq \limsup_{i\rightarrow \infty}\frac{\phi_{n_i}(x)}{n_i}\leq\lim_{i\rightarrow \infty} \int \psi_kd\nu_{i}+\frac{1}{k}. \]

	And as $\nu_{i}\rightarrow \mu_2$  we have
	
	\begin{eqnarray}\label{eq 1}
	\int \psi_kd\mu_2-\frac{1}{k}\leq \liminf_{i\rightarrow \infty}\frac{\phi_{n_i}(x)}{n_i} \leq \limsup_{i\rightarrow \infty}\frac{\phi_{n_i}(x)}{n_i}\leq\int \psi_kd\mu_2+\frac{1}{k}. 
	\end{eqnarray}
	In the Equation (\ref{eq 1}), we have
	
	\[\int \psi_kd\mu_2-\frac{1}{k}\leq \limsup_{i\rightarrow \infty}\int\frac{\phi_{n_i}(x)}{n_i}d\mu_{2}=\Phi^*(\mu_2)\leq\int \psi_kd\mu_2+\frac{1}{k}.\]
	Therefore, as $k$ is arbitrary, we have
	$\limsup_{i\rightarrow \infty}\frac{\phi_{n_i}(x)}{n_i}=\Phi^{*}(\mu_2)\neq \Phi^*(\mu_1)$.

	Conversely, if for $\Phi \in\mathfrak{C}^0(f,X)$ and $x \in X$ we have
	\[\inf_{\mu \in \mathcal{M}_xf}\Phi^{*}(\mu)< \sup_{\mu \in \mathcal{M}_xf}\Phi^{*}(\mu),\]
	Then we can make two measures $\mu_1$ and $\mu_2$ in $\mathcal{M}_x(f)$ such that 
	$\Phi^{*}(\mu_1)< \Phi^{*}(\mu_2)$ and two convergence subsequences 
	
	\[\frac{1}{n_i}\sum_{i=0}^{n_i-1}\delta_{f^j(x)}\rightarrow \mu_1 \;\; \mbox{and}\;\; \frac{1}{m_i}\sum_{i=0}^{m_i-1}\delta_{f^j(x)}\rightarrow \mu_2.\] Then, as $\Phi$ is asymptotically additive, for each $k$, there exists $\varphi_{1/k}:=\varphi_k \in C^{0}(X)$ such that, for all $x\in X$ and large $n$ we have
	$\left\lvert\frac{1}{n}\phi_{n}(x)-\frac{1}{n}S_{n}\varphi_k(x)\right\rvert<\frac{1}{k}<\varepsilon$. Moreover, we can obtain that $\left\lvert\int\varphi_kd\mu_1-\Phi^{*}(\mu_1)\right\rvert<\frac{1}{k}<\varepsilon$. Then, for $i$ sufficiently large
	
	\begin{eqnarray*}
		\left\lvert\frac{1}{n_i}\phi_{n_i}-\Phi^{*}(\mu_1)\right\rvert&\leq& \left\lvert\frac{1}{n_i}\phi_{n_i}(x)-\frac{1}{n_i}S_{n_i}\varphi_k(x)\right\rvert  +\left\lvert\frac{1}{n_i}S_{n_i}\varphi_k(x)-\int \varphi_k d\mu_1\right\rvert\\
		&+&\left\lvert\int \varphi_k d\mu_1-\Phi^{*}(\mu_1)\right\rvert
		< 3\varepsilon.
	\end{eqnarray*}
	Then, $\lim_{i\rightarrow \infty}\frac{1}{n_i}\phi_{n_i}(x)=\Phi^{*}(\mu_1)$. Analogously, $\lim_{i\rightarrow \infty}\frac{1}{m_i}\phi_{m_i}(x)=\Phi^{*}(\mu_2)$. Therefore, $x \in I(\Phi,f)$ and $\Phi \in \hat{\mathfrak{C}}^0(f,X) $.

\end{proof}



\textbf{\textit{Proof of the Theorem B}}
\begin{proof}
	We fix $\varepsilon>0$ and $\Phi \in \mathfrak{\hat{C}}^0(f,X)$. By Variational Principle \ref{eq variational principle}, we choose a ergodic measure $\mu$ such that $h_{\mu}(f)+\Phi^{*}(\mu)>P_f(\Phi)-\varepsilon$. Choose $\theta\in (0,1)$ close to 1 satisfying
	\[\theta(h_{\mu}(f)+\Phi^{*}(\mu))>P_f(\Phi)-\varepsilon,\]
	and $(1-\theta)\|\Phi\|<\varepsilon$, where $\|\Phi\|$ is a positive constant such that $\|\Phi\|\geq \Phi^{*}(\mu) $ for all $\mu \in \mathcal{M}_fX$, see item 2, Proposition A1 of \cite{Feng2010}. For $\Upsilon \in D$, by Lemma \ref{lemma.inf.sup}, there is an invariant measure $\mu_{\Upsilon}$ such that $\Upsilon^{*}(\mu_{\Upsilon})\neq \Upsilon^{*}(\mu)$. Take the measure $\nu_{\Upsilon}=\theta \mu+(1-\theta)\mu_{\Upsilon}$. Then
	
	\begin{eqnarray*}
		h_{\nu_{\Upsilon}}(f)+\Phi^{*}(\nu_{\Upsilon})&=&\theta h_{\mu}(f)+(1-\theta)h_{\mu_{\Upsilon}}(f)+\theta\Phi^{*}(\mu)+(1-\theta)\Phi^{*}(\mu_{\Upsilon})\\
		&\geq& \theta(h_{\mu}(f)+\Phi^{*}(\mu))-(1-\theta)\|\Phi\|> P_f(\Phi)-2\varepsilon.
	\end{eqnarray*}
	Remember that $\Phi^{*}(\mu)=\displaystyle
	\lim_{n\rightarrow \infty}\int \frac{\phi_n(x)}{n}\mu(x)$. Then, by product topology, two  asymptotically additive sequence $\Psi:=(\psi_n)_n$ and $\Upsilon:=(\upsilon_n)_n$ are close if the sequences of continuous functions $\displaystyle\frac{\psi_n}{n}$ and $\displaystyle\frac{\upsilon_n}{n}$ are close for each $n$. Therefore, by continuity of sequences, for each $\Upsilon \in D$ there exist an neighborhood $\mathcal{V}_{\Upsilon}\subseteq \mathfrak{C}^0(f,X) $ of $\Upsilon$ such that for each $\Psi \in \mathcal{V}_{\Upsilon}$ we have $\Psi^{*}(\mu)\neq \Psi^{*}(\nu_{\Upsilon})$. So $\{\mathcal{V}_{\Upsilon}: \Upsilon \in D\}$ forms an cover of $D$. We can take a countable subcover $\{\mathcal{V}_{\Upsilon_i}\}_{i\in \mathbb{N}}$ of $D$ because $C^0(X)$ has countable topological basis and then the product topological  $\mathfrak{C}^0(f,X)$ also has(second axiom of countability). Then, for each $\Psi \in D$, there is $i\geq 1$ such that $\Psi \in \mathcal{V}_{\Upsilon_i}$ and satisfying
	\[\Psi^{*}(\mu)\neq \Psi^{*}(\nu_{\Upsilon_i}).\]
	
	Put $\nu_i:= \nu_{\Upsilon_i}$ and define the sequence of measures $\eta_i:=\theta_i\mu+(1-\theta_i)\nu_i$, where $(\theta_i)_i\subset (0,1)$ is a increasing sequence converging the 1 such that $\theta_1\geq \theta$. Then, for each $\Psi\in D$, there is $i\geq 1$ such that $ \Psi^{*}(\mu)\neq \Psi^{*}(\eta_{i})$. Note that
	
	\begin{eqnarray*}
		h_{\eta_i}(f)+\Phi^{*}(\eta_{i})&\geq& \theta_i(h_{\mu}(f)+\Phi^{*}(\mu))+(1-\theta_i)h_{\nu_i}(f)-(1-\theta_i)\|\Phi\|\\
		&\geq& \min\left\{\theta(h_{\mu}(f)+\Phi^{*}(\mu)),h_{\mu}(f)+\Phi^{*}(\mu)\right\}-(1-\theta)\|\Phi\|\\
		&>&P_f(\Phi)-2\varepsilon.
	\end{eqnarray*}
	
	Define the set $K:=\{\mu\}\cup \bigcup_{i=1}^{\infty}\{t\eta_i+(1-t)\eta_{i+1}, t \in [0,1]\}$. By fact that $\eta_i$ converge for $\mu \in K$ this implies that $K$ is compact and connected set and that every $\nu \in K$ satisfies $h_{\nu}(f)+\Phi^{*}(\nu)>P_f(\Phi)-2\varepsilon$.
	Since $f$ is sequential saturated, then
	\[P_f(G_K,\Phi)=\inf\{h_{\nu}(f)+\Phi^{*}(\nu): \nu \in K\}\geq P_f(\Phi)-2\varepsilon.\]
	For finishing the proof of Theorem, we show that $G_K\subseteq \bigcap_{\Psi \in D}I(\Psi,f)$. Let $x\in G_K$ and $\Psi \in D$. Then $\Psi^{*}(\mu)\neq \Psi^{*}(\eta_i)$ for some $i\geq 1$. We have that $\mathcal{M}_xf=K\supseteq \{\mu, \eta_i\}$. Then, there two sequences $(n_j)_j$ and $(m_j)_j$ such that $\frac{1}{n_j}\sum_{l=0}^{n_j-1}\delta_{f^l(x)}=\mu$ and $\frac{1}{m_j}\sum_{l=0}^{m_j-1}\delta_{f^l(x)}=\eta_i$. And as $\Psi$ is asymptotically additive, we have 
	
	\begin{eqnarray*}
		\left\lvert\frac{1}{n_i}\psi_{n_i}(x)-\Psi^{*}(\mu)\right\rvert&\leq& \left\lvert\frac{1}{n_i}\psi_{n_i}(x)-\frac{1}{n_i}S_{n_i}\varphi_k(x)\right\rvert + \left\lvert\frac{1}{n_i}S_{n_i}\varphi_k(x)-\int \varphi_k d\mu_1\right\rvert\\
		&+&\left\lvert\int \varphi_k d\mu_1-\Psi^{*}(\mu)\right\rvert
		< 3\varepsilon.
	\end{eqnarray*}
	Therefore $\lim_{i\rightarrow \infty}\frac{1}{n_i}\psi_{n_i}(x)=\Psi^{*}(\mu_1)$. Analogously, $\lim_{i\rightarrow \infty}\frac{1}{m_i}\psi_{n_i}(x)=\Psi^{*}(\eta_i)$. Then, $x \in I(\Psi,f)$. As $x \in G_K$ and $\Psi \in D$ are arbitrary, the Theorem it's proved.
	
\end{proof}

\section{Applications and examples}\label{aplication}
In this section, we given some examples and applications of our results.
\begin{example}
	Our result apply to the  asymptotically additive sequences  of the examples \ref{Ban-Cao-Hu2010} and \ref{Barreira1996}. 
	
	
	
\end{example}

\subsection{Cocycles  under shift of the finite type}

We will consider cocycles under shift of the finite type in \cite{FengKaenmaki2011}. For more details, see \cite{Barreira-survey}.

\begin{example}
	Let $\sigma:\Sigma \rightarrow \Sigma$ be the shift map on the space $\Sigma=\{1,...,m\}^{\mathbb{N}}$, $m\geq 2$ doted with the metric $d(x,y)=2^{-\min\{j\geq 1;x_j\neq y_j\}}$ where $x=(x_1x_2...)$ and $y=(y_1y_2...)$.
	
	Consider matrices $M_1,...,M_m \in \mathcal{M}_{d\times d}(\mathbb{C})$ such that for each $n\geq 1$ there exists $i_1,...,i_n \in \{1,...,m\}^{n}$ such that $M_{i_1}\cdot ... \cdot M_{i_n}\neq 0$. Then, the topological pressure function is well defined with  
	\[P(t)=\lim_{n\rightarrow \infty}\frac{1}{n}\log \sum_{w\in \{1,...,m\}^{n}}\|M_{i_1}\cdot ... \cdot M_{i_n}\|^t,\]
	where $w=(i_1,...,i_n)$.
	
	We define a class of functions that are obtained via a product of matrices. For each $t\geq 0$, $n \in \mathbb{N}$ and $w=(i_1,...,i_n) \in \{1,...,m\}^{n}$, we consider the locally constant functions $\psi^t_{w}:\Sigma \rightarrow \mathbb{R}^+$
	\[\psi^t_{w}(x)=\|M_{i_1}\cdot ... \cdot M_{i_n}\|^t.\] 
	We define a sequence of functions $\Phi^t=(\phi^t_n:\Sigma \rightarrow \mathbb{R})_n$ of the following form
	\[\phi^t_n(x)=\sup_{w'\in C(x)}\log \psi_{w'}^t(x)=\sup_{w'\in C(x)}\log \|M_{i_1}\cdot ... \cdot M_{i_n}\|^t,\]
	where $w'=(i_1,...,i_n)$ and $C(x)$	is the set of blocks of $n$ elements  that are equal to the first elements of $x$. 

	In \cite{Feng2009}, it show that there exists $C>0$ and $k\in \mathbb{N}$  such that $w, w' \in \bigcup_{n\in \mathbb{N}}\{1,...,m\}^n$ there exists $\overline{w}\in \sum_{j=1}^{k}\{1,...,m\}^k$ for which
	
	\begin{eqnarray}\label{condition distortion}
	\|M_w M_{\overline{w}} M_{w'}\|\geq C \|M_w\|\cdot \|M_{w'}\|.
	\end{eqnarray}
	The property in (\ref{condition distortion}) ensures that the sequence $\Phi^t$ is almost additive(see \cite{Feng2009},\cite{Barreira-survey}), then is asymptotically additive. Then, by Theorem B, we have
	\[P_{\sigma}\left(\bigcap_{s\geq 0}I(\Phi^s),\Phi^t\right)=P_{\sigma}(\Phi^t)\]
	for each $t\geq0$.
	
\end{example}

\subsection{Nonconformal repellers}

We describe a class of nonconformal repellers considered by Barreira and Gelfert in \cite{Barreira-Gelfert}. For more details, see also \cite{Barreira-survey}. 

Let $f:\mathbb{R}^2\rightarrow \mathbb{R}^2$ be a $C^1$ map and let $\Lambda\subset \mathbb{R}^2$ a compact $f$-invariant set. We say that $f$ is expanding map on $\Lambda$ and that $\Lambda$ is a repeller of $f$ if there exist constant $C>0$ and $\beta>1$ such that 
\[\|Df^n(x)v\geq C \beta^n\|v\|\]
for each $x\in \Lambda, n \in \mathbb{N}$ and $v\in T_xM$.
We assume that there a open set $\mathcal{U}\supset \Lambda$ such that $\Lambda=\bigcap_{n\in \mathbb{N}}f^n(\mathcal{U})$ and that $f$ is  topologically mixing on $\Lambda$.

Given a number $\gamma<\frac{1}{2}$ and a subspace 1-dimensional $E(x)\subset \mathbb{R}^2$, we considered the cone

\[C_{\gamma}(x):=\{(u,v)\in E(x)\oplus E(x)^{\perp};\|v\|\leq \gamma \|u\|\}.\]

We say that differential map $f:\mathbb{R}^2\rightarrow \mathbb{R}^2$ satisfies the ``cone condition" on a compact set $\Lambda \subset \mathbb{R}^2$ if there exist $\gamma<1$
and for each $x \in \Lambda$ a 1-dimensional subspace $E(x)\subset \mathbb{R}^2$ varying continuously with $x$ such that

\[Df(x)C_{\gamma}(x)\subset\{0\}\cup \mbox{int}C_{\gamma}(f(x)).\]

Let $\Phi_i$ be the almost additive sequence obtained as follows: Let the singular values of a $2\times 2$ matrix $A$, and

\[\sigma_1(A)=\|A\|\;\;\; \mbox{and}\;\;\; \sigma_2(A)=\|A^{-1}\|^{-1}.\]

Given a $C^1$ map $f:\mathbb{R}^2\rightarrow \mathbb{R}^2$. The sequences of functions $\Phi_i=(\phi_{i,n})_n$, $i=1,2$, is given by
\[\phi_{i,n}(x)=\log \sigma_i(Df^n(x)).\]
\begin{proposition}[\cite{Barreira-Gelfert}, Proposition 4]
	Let $\Lambda$ be a repeller of a $C^1$ map $f:\mathbb{R}^2\rightarrow \mathbb{R}^2$. If $f$ satisfies a cone condition on $\Lambda$, then $\Phi_i$ is almost additive sequence for $i=1,2$. 
\end{proposition}
Let $\delta>0$ be such that for each $x\in \Lambda$ the map $f$ is invertible on $B(x,\delta)$. For each $x\in \Lambda$ and $n\in \mathbb{N}$ we define
\[B(x,n,\delta)=\bigcap_{l=0}^{n-1}f^{-l}(B(f^l(x),\delta).\]
We say that $f$ have \textit{ bounded distortion}  on $\Lambda$ if there exist $\delta>0$ such that

\[\sup\left\{\|Df^n(y)(Df^n(z))^{-1}\|;x \in \Lambda\;\; \mbox{and}\;\; y,z \in B(x,n,\delta)\right\}< \infty.\]

We will see the relationship between the sequences $\Phi_i$ and the Lyapunov exponents. Given a differentiable transformation $f:M\rightarrow M$. By Oseledet's multiplicative ergodic theorem, for each finite $f$-invariant measure $\mu$ on $M$ there exist a $\mu$-full set $X\subset M$  such that $x\in X$, there exists numbers $\lambda_1(x),...,\lambda_{s(x)}$ and subspaces $M=V_1(x)\supset V_2(x)\supset...\supset V_{s(x)}(x)\supset V_{s(x)+1}(x)=\{0\}$ such that
\[\lim_{n\rightarrow \infty}\frac{1}{n}\log\|Df^n(x)v\|=\lambda_i(x),\]
for each $v\in V_{i}(x)\backslash V_{i-1}(x)$ and $i=1,...,s(x)$. The numbers $\lambda_1(x),...,\lambda_{s(x)}$ are the Lyapunov exponents of $\mu$.
In particular, for $f$ above($M=\mathbb{R}^2$) and for $x\in X$ we have

\[\lim_{n\rightarrow \infty}\frac{\phi_{i,n}(x)}{n}=\lim_{n\rightarrow \infty}\frac{1}{n}\log\sigma_i(Df^n(x))=\lambda_i(x)\;\;\;\mbox{for each}\;\;\; i=1,2.\]
Barreira and Gelfert, \cite{Barreira-Gelfert}, proved that if a $C^1$ map $f:\mathbb{R}\rightarrow \mathbb{R}$ has bounded distortion and satisfy a cone condition on $\Lambda$ then $(f,\Phi_i)$ has a unique equilibrium states $\mu_i$. Then, by Theorem B, ensures that
\begin{eqnarray*}
	P_f\left(I(\Phi_1)\cap I(\Phi_2),\Phi_i\right)
	&=&P_f(\Phi_i)=h_{\mu_i}(f)+\int_{\Lambda}\lambda_i(x)d\mu_i(x)\\
	&=&h_{\mu_i}(f)+\lim_{n\rightarrow \infty}\frac{1}{n}\int_{\Lambda}\log \sigma_i(Df^n(x))d\mu_i(x)
\end{eqnarray*}
for each $i=1,2$.

\hspace{1.6cm}

\textbf{Acknowledgments}: This work was partially supported by a INCT-Mat/Capes-Brazil postdoctoral
fellowship at University of Bahia. The author is grateful to P.
Varandas  and V. Ramos for reading a preliminary version of this paper and making important suggestions that helped improve the presentation of the text.  The author  is also grateful to anonymous referees for giving us a simpler proof for the Theorem A and suggestions that will help improve the text.
\vspace{2cm}

\end{document}